\documentclass[10pt]{article}

\usepackage{amsmath}
\usepackage{amsthm}
\usepackage{amssymb}
\usepackage{verbatim}
\usepackage[dvips]{graphicx}
\usepackage{enumerate}
\usepackage{t1enc}
\usepackage[latin2]{inputenc}
\usepackage[english]{babel}
\usepackage{color}
\usepackage[margin=6mm,left=3cm,right=3cm,top=1cm,bottom=2cm,nohead]{geometry}
\theoremstyle{plain}
\newtheorem{theorem}{Theorem}
\newtheorem{lemma}{Lemma}
\newtheorem{proposition}{Proposition}
\newtheorem{corollary}{Corollary}

\theoremstyle{definition}

\theoremstyle{remark}

\theoremstyle{definition}

\title{On Linkedness of the Cartesian Product of Graphs}
\author{G\'abor M\'esz\'aros\footnote{Research was supported by the Balassi Institute, the Fulbright Commission, and the Rosztoczy Foundation.}}
\begin{document}
\maketitle
\begin{abstract}
We study linkedness of the Cartesian product of graphs and prove that the product of an $a$-linked and a $b$-linked graphs is $(a+b-1)$-linked if the graphs are sufficiently large. Further bounds in terms of connectivity are shown. We determine linkedness of products of paths and products of cycles.
\end{abstract}
\section*{Introduction}

Throughout this paper we use the notation of \cite{modern}. For the sake of completeness we recall definitions of the mainly used concepts. The {\it connectivity} of a simple graph $G = (V(G), E(G))$ (denoted by $\kappa(G)$) is the smallest number of vertices whose removal from G results in a disconnected graph or a graph of one vertex. The {\it Cartesian product} of graphs $G$ and $H$ is the graph $G\square H$ with vertices $V(G\square H)=V(G)\times V(H)$, and $(x,u)(y,v)$ is an edge  if $x=y$ and $uv\in E(H)$ or $xy\in E(G)$ and $u=v$. Product of graphs $G_1,\dots,G_t$ for $t\geq 3$ is defined recursively. Note that the Cartesian product is an associative operation. The graphs $G_1,\dots,G_t$ are called factors of $G_1\square\dots\square G_t$. The Cartesian product is a well studied graph product and it gave rise to important classes of graphs; for example, the $n$-dimensional grid can be considered as the Cartesian product of lower dimensional grids. Hypercubes are well known members of this family with similar recursive structure:  the Cartesian product of an $m$-dimensional hypercube and 
an $n$-dimensional hypercube is an $(m+n)$-dimensional one.  

The study of graph products leads deep structural problems such as invariance and inheritance of graph parameters: connections between parameters of products and their factors have been extensively studied. Note that among the several graph products (see \cite{product}) the Cartesian product is also known as {\it direct sum} referring to the fact that many of the classical graph parameters inherit additively. In case of {\it minimum, maximum} and {\it average degree} it can be showed easily that  $\delta(G\square H)=\delta(G)+\delta(H)$, $\Delta(G\square H)=\Delta(G)+\Delta(H)$ and $\underline{d}(G\square H)=\underline{d}(G)+\underline{d}(H)$. We present some further results with linear bounds. 
Chiue and Shieh \cite{conn ineq} proved that the Cartesian product of a $k$-connected and an $l$-connected graph is $(k+l)$-connected. Later on, Spacapan  \cite{conn} determined the connectivity number of $G\square H$, namely $\kappa(G\square H)=\min\big(\delta(G)+\delta(H), \kappa(G)\cdot |V(H)|,\kappa(H)\cdot |V(G)|\big)$.
Gy\H ori and Plummer \cite{Gyori Plummer} proved  that the Cartesian product of a $k$-extendable and an $l$-extendable graph is $(k+l+1)$-extendable (a graph $G$ is $k$-extendable if $G$ is connected, has a perfect matching and any matching of $k$ edges in $G$ can be extended to a perfect matching).
 
In this paper we study linkedness of the Cartesian product of graphs. Menger's theorem (see \cite{modern}) implies that a graph is $k$-connected if and only if for every (not necessarily disjoint) $k$-tuples $S=\{s_1,\dots  , s_k\}$ and $T=\{t_1,\dots, t_k\}$ there exist disjoint paths $P_1,\dots, P_k$ joining every $s_i$ to $t_{\pi(i)}$ for some $\pi\in S_k$. Menger's theorem provides no control on the actual pairing of $S$ and $T$ via paths. A graph $G$ is $k$-{\it linked} if, for every ordered set of $2k$ vertices $S=(s_1,\dots, s_k)$ and $T=(t_1,\dots, t_k)$ there exist internally disjoint paths $P_1,\dots, P_k$ such that each $P_i$ is an $s_i, t_i$-path. It is a well-known but somewhat technical result that a graph $G$ is $k$-linked if and only if the above condition holds for every choice of disjoint $S$ and $T$ sets. We use the proposition without proof and always assume that $S$ and $T$ are disjoint for the sake of simplicity. We use the notation $link(G)$ for the linkedness-number of a graph $G$, that is, the largest positive integer $k$ for which $G$ is $k$-linked.

Linkedness is a natural strenghtening of connectivity. It is easy to see that $k$-linked graphs are  $(2k-1)$-connected. Certainly, placing vertices $s_1,t_1,\dots,s_{k-1}, t_{k-1}$ in a graph $G$ to a cut $D$ of size $(2k-2)$ makes impossible to join $s_k$ to $t_k$ if they are located in different components of $G-D$. It has been also known for some time that sufficient connectivity would imply linkedness. Bollob\'as and Thomason \cite{Bollobas} gave the first linear upper bound proving that $22k$-connected graphs are $k$-linked. This bound has been improved to $10k$ by Thomas and Wollan \cite{10k} and it is also very likely that the connectivity needed to imply k-linkedness is significantly less than $10k$. When girth conditions are placed on the graph, then almost sharp results between connectivity and linkage can be proven. Mader \cite{Mader} proved that $2k$ connected graphs with sufficiently large girth are $k$-linked. The condition on the girth has been weakened by Kawarabayshi \cite{girth}. Note that neither Mader's nor Kawarabayshi's result can be applied for the Cartesian Product of graphs, as of the Cartesian Product of two nonempty graphs is upperly bounded by four.

In this paper, we prove that the Cartesian product of an $a$-linked graph and a $b$-linked graph is $(a+b-1)$ linked if the graphs are sufficiently large.
\begin{theorem}\label{a+b-1}
If $G$ is an $a$-linked graph with $|V(G)|\geq 8a$ and $H$ is a $b$-linked graph with $|V(H)|\geq 8b$ then $ G\square H$ is $(a+b-1)$-linked.
\end{theorem}

Remark that the bound in Theorem \ref{a+b-1} is sharp. Let $n,k\in \mathbb{Z}^+$, $n\geq (2k-1)$ and construct a graph $G$ as follows: take the complete graph $K_n$ on $n$ vertices and an additional vertex that is adjacent to $(2k-1)$ vertices in $K_n$. Easy to see that $G$ is $k$-linked, $(2k-1)$-connected, while $G\square G$ is $(4k-2)$-connected, hence it cannot be $2k$-linked as $2k$-linked graphs are $(4k-1)$-connected. As $n$ does not depend on the choice of $k$ (only $n\geq 2k$ is required) it provides and infinite family of products where equality holds. Later on we prove that higher connectivity of $G$ (with all other settings unchanged) yields better lower bound on the linkedness of the product graph. 

It follows from Theorem \ref{a+b-1} that the product $G\square H$ of a $k$-linked graph $G$ and graph $H$ is also $k$-linked if $H$ is connected, while disconnected $H$ makes $G\square H$ also disconnected. In the second part of the paper we find sufficient conditions for a graph $H$ such that the product $G\square H$ of $H$ and a (sufficiently large) $k$-linked graph $G$ is $(k+1)$-linked. Using that theorem, in the last section we determine linkedness of products of paths and products of cycles. 
\begin{theorem}\label{k->k+1}
If $G$ is a $k$-linked graph with $k\geq 2$ and $|G|\geq \max(9, 4k)$ and $H$ is a 2-connected graph then $G\square H$ is $(k+1)$-linked.
\end{theorem}

Before the proofs we fix further terminologies and notation. 
A $G$-{\it layer} $G_x$ ($x\in V(H)$) of the Cartesian product $G\square H$ is the subgraph induced by the set of vertices $\{(u,x): u\in V(G)\}$. An $H$-layer is defined analogously. We call edges of $G\square H$ lying in $G$-layers {\it vertical} while edges lying in $H$-layers are called {\it horizontal}. Unless misleading we also use the notation $G_z=G_x$ and $H_z=H_y$ for layers corresponding to $z=(x,y)\in G\square H$. The {\it projection} of vertex $(u,v)$ to a horizontal layer $G_x$ or a vertical layer $H_y$ is $(u,x)$ and $(y,v)$, respectively.
The set of neighbours of a vertex $x$ in a graph is denoted by $\Gamma_G(x)$.  
 For a graph $G$ and for a vertex $x\in V(G)$ or a 
subgraph $H \subset G$ we use the notation $G-x$ and $G-H$ for $G\backslash\{x\}$ and $G\backslash H$, respectively. The size of a set $H$ is denoted by $|H|$.
The labelled vertices $S=(s_1,\dots, s_k)$ and $T=(t_1,\dots, t_k)$ to be linked are sometimes called {\it terminals}, the sets $\{u_i, v_i\}$ are {\it pairs} or {\it matching terminals}. Finding a path for a pair is often called {\it joining} the pair.

\section*{Proof of Theorem \ref{a+b-1}}
Recall a straightforward corollary of Menger's Theorem:
\begin{lemma}\label{classic}
If $G$ is $k$-connected, $m,n\in \mathbb{Z}^+$, $m+n\leq k$, then for every disjoint tuples $D=\{d_1,\dots  , d_m\}$, $S=\{s_1,\dots  , s_n\}$ and $T=\{t_1,\dots, t_n\}$ there exist disjoint paths $P_1,\dots, P_n$ in $G-D$ joining every $s_i$ to $t_{\pi(i)}$ for some $\pi\in S_n$.
\end{lemma}
\begin{proof}
Use Menger's Theorem on sets $S$ and $T$ in the graph $G-D$.
\end{proof}
We first settle the case when $a$ or $b$ is equal to 1. Note that being 1-linked is equivalent to connectivity.

\begin{lemma}\label{b=1}
Let $G$ be $k$-linked, $H$ be connected. Then $G\square H$ is $k$-linked as well.
\end{lemma}

\begin{proof} Let $M$ denote the set of $2k$ (arbitrarily chosen and paired) terminals in $G\square H$. Take a $G$-layer $G_x$ $(x\in H)$ with terminals $u_1,\dots,u_t$ ($1\leq t \leq 2k$). If $t=2k$, use the condition that $G_x$ is $k$-linked and find the necessary paths within the layer. Otherwise, let $D= \{u_1,   \dots, u_t\}$,  $S= M-D$ and let $T$ consist of $(2k-t)$ non-terminal vertices in $G_x$. Using Lemma \ref{classic} one can find $(2k-t)$ paths from $S$ to $T$. For such a path $P$, starting at terminal $t$ in $S$, let $t'$ denote the first vertex of $P$ in $G_x$ (the vertex where $P$ first "enters" $G_x$). Truncate $P$ to a $t \-- t' $ path for every choice of $t$. Using the condition that $G_x$ is $k$-linked, one can find $k$ paths $Q_1,\dots, Q_k$ that join the $2k$ vertices of the set $D\cup \{t': t\in S\}$, with the obvious matching ($t'$ gets the original pair of $t$). The path $Q_1,\dots, Q_k$ extended by paths $P_1,\dots, P_{2k-t}$ at vertices ${t': t\in S}$ form an appropritate path system for the initial matching.  
\end{proof}

From now on, we may assume $a\geq b\geq 2$. We prove a more general form of Theorem \ref{a+b-1}:
\begin{theorem}\label{strong}
If $G$ is an $a$-linked graph with $|G|\geq 8a$ and $H$ is a $(2b-1)$-connected graph with $a\geq b $ then $G\square H$ is $(a+b-1)$-linked.
\end{theorem}
\begin{proof} 
Our main goal in the proof is to carry out one of the following tasks.

\begin{enumerate}
\item[i)]
Join one terminal to its pair within a layer and proceed by induction on an appropriate subgraph.
\item[ii)]
For every pair $(x,y)$ find paths $P_x$, $P_y$ with other endvertices $x'$ and $y'$,
such that $x'$ and $y'$ share the same horizontal layer. Following that we will find a path $Q$ joining $x'$ and $y'$ and join $x$ and $y$ by the concatenation $P_x \-- Q \-- P_y$.
\end{enumerate}
For the latter task, observe that, as the total number of terminals is $(2a+2b-2)$ and $a\geq b$, two approriate horizontal $G$-layers will be sufficient to contain and join all the $x'$-s and $y'$-s. 
The bottleneck of the idea is that all the $P_x$, $P_y$ paths have to be disjoint. We also want to make sure that these paths enter only one of the above distinguished horizontal layers containing the $x'$-s and $y'$-s. We will use Lemma \ref{classic} to guarantee such conditions. 
We call a $G$-layer {\it crowded} if it contains more than $(2a-1)$ terminals. Observe that crowded G-layers necessarily contain at least one pair of matching terminals.

If there exists a crowded $G$-layer $G_x$ $(x\in H)$ in $G\square H$, take a pair $u_1,v_1\in G_x$. As $|\Gamma_H(x)|\geq (2b-1)$, there exists $y\in \Gamma_H(x)$ such that $G_y$ contains no terminal. The appropriate neighbours of $u_1$ and $v_1$ 
in $G_y$ can be joined by a path within $G_y$. We can join $u_1$ and $v_1$ by extending that path on both ends by the vertical edges from $u_1$ and from $v_1$ to $G_y$. For every remaining terminal $u$ of $G_x$ we find a vertical neighbour not belonging to $G_y$ as follows (note that case i) and case ii) do not exclude each other).

\begin{itemize}
\item[i)] Link $u$ to its pair if they are adjacent by a vertical edge.
\item[ii)] If the terminal $u$ has a vertical neighbour $u'$ that is neither a terminal nor has it been previously assigned as a vertical neighbour to another terminal in $G_x$, choose $u'$.
\item[iii)] If neither of the previous cases applies, then $H_{u}$ contains all terminals lying outside of $G_x$ and its pair $v$ lies in $G_x$. Switch $(u_1,v_1)$ to $(u_2,v_2)$ and start again. The second round terminates without encountering the same problem.  
\end{itemize}
Define a new pairing of the remaining $(a+b-2)$ pairs of terminals by substituting every $u$ by $u'$. Observe, that $G$ and $H-x-y$ are $a$-linked and $(b-1)$-linked and have at least $8a$ and $(8b-2)$ vertices, respectively. By inductional hypothesis, $G\square (H-x-y)$ is $(a+b-2)$ linked and so there exist $(a+b-2)$ paths joining the newly defined $(a+b-2)$ pairs. The extension of these paths by the appropriate $\overline{uu'}$ edges results in a path system that joins the original pairing.  

Assume now that $G\square H$ contains no crowded $G$-layer. For a terminal $u$ our first goal is to find a path with horizontal edges to a vertex $u'\in G_u$ such that $H_{u'}$ is devoid of terminals and vertices of previously routed paths of the same kind. We carry out this task in several rounds, defining a $u'$ vertex and a corresponding $u-u'$ path for every  $u$ terminal of a given $G$-layer within a round.  
As long as the number of terminals on layers being or having been processed does not exceed $(2a-1)$, Lemma \ref{classic} provides an easy way of assignment. As Menger's theorem itself does not provide any control on the length of the joining paths, our proof will frequently use the following {\it truncation} operation. Assume we are given a path $P$ of horizontal edges with a terminal end $u$ and a non-terminal endvertex $\hat{u}$, whose $H_{\hat{u}}$ layer does not contain terminals or vertices of previously defined paths. Starting with $u$, we read the vertices of $P$ in precedence order until we find the first vertex $u'$ that has the same properties as $\hat{u}$. We stop and truncate $P$ to an $u \-- u'$ path. 

Consider all $G$-layers $G_1,\dots, G_n$ containing $0<s_1\leq\dots\leq s_n<2a$ terminals. Choose $1\leq t\leq n$ such that $\sum\limits_{i=1}^{t-1}s_i\leq (2a-1)$ and $\sum\limits_{i=1}^{t}s_i> (2a-1)$. We design our algorithm as follows:
\begin{enumerate}
\item[i)] In round 1, choose a set of $s_1$ vertices in $G_1$ whose corrensponding $H$-layers do not contain any terminal. Use Menger's theorem to find $s_1$ disjoint paths between the terminals of $G_1$ and the set. Truncate these paths and define the set $D$ as the set of the non-terminal endpoints of the truncated paths. 
\item[ii)] I round $i$ for $2\leq i\leq (t-1)$, let $T$ denote the set of terminals in $G_i$ and let $D_i$ be the projection of $D_{i-1}$ to $G_i$. Choose a set $S$ of $s_i$ vertices in $G_i$ whose corrensponding $H$-layers do not contain any terminal or vertex of $D_i$.  Easy to see that $|D_{i}|=\sum\limits_{j=1}^{i}s_j\leq (2a-1)$, hence the conditions of Lemma \ref{classic} hold. Take $s_i$ paths joining (in some order) $S$ and $T$. Truncate the paths and update $D_i$ by adding the set of the paths's non-terminal endpoints. 

\item[iii)] In the remaining $(n-t+1)$ rounds ($t\leq i \leq n$), choose a set of $s_i$ vertices in $G_j$ whose corrensponding $H$-layers do not contain any terminal. Use Menger's Theorem to find $s_i$ disjoint paths joining (in some order) the terminals and the newly chosen vertices. 
\end{enumerate}

We refer to the previous phase as a {\it global horizontal shift}. Observe that each terminal $u$ was given a non-terminal vertex $u'\in G_u$ and an $uu'$ path $P_{uu'}$ of horizontal edges, such that:

\begin{enumerate}
\item[A)] $P_{uu'}$ does not intersect with other paths defined in the phase.
\item[B)] $H_{u'}$ consist of at most $(n-t+1)$ vertices belonging to other paths defined in the phase (at most one at each layer during the last $(n-t+1)$ steps). 
\end{enumerate}

Note that the condition $V(G)\geq 8a$ guarantees that every step of the horizontal shift can be carried out without running out of space. Our next goal is to carry out a {\it global vertical shift}. We take
two $G$-layers that contain neither terminals nor vertices belonging to paths of the previous phase and call them $G_\alpha$ and $G_\beta$. If no such layers are available, let $G_\alpha=G_1$, $G_\beta=G_2$ and skip Round 1 and 2 in the previous pahse. Note that $|V(H)|\geq 2b$, hence neither of $G_1$ and $G_2$ contains more than $a$ terminals. For each $u'$ of the previous phase we define a vertex $u''$ and a $u'-u''$ path in $H_{u'}$ such that:
\begin{enumerate}
\item[i)] $u''\in G_\alpha$ or $u''\in G_\beta$,
\item[ii)] if $(u,v)$ are a pair, then $u''$ and $v''$ belong to the same $G$-layer,
\item[iii)] $G_\alpha$ and $G_\beta$ both have at most $a$ pairs of $(u'',v'')$ vertices, 
\item [iv)] the path $P_{u'u''}$ does not intersect other paths of the recent or the previous phase (with the exception of $P_{uu'}$). In addition, if $u''\in G_\alpha$, then $P_{u'u''}\cap G_\beta =\emptyset$, if $u''\in G_\beta$, then $P_{u'u''}\cap G_\alpha =\emptyset$. 
\end{enumerate}

Clearly, $G_\alpha$ and $G_\beta$ will provide room for the final step of joining the terminals. As both layers are $a$-linked, all $(u'',v'')$ pairs can be joined by disjoint path. Our initial pair $(u,v)$ will be joined by an $u\-- u'\--u''\--v''\--v'\-- v$ path. It remains to show that the $P_{u'u'}$ can be found with the above conditions. Distribute the $(u,v)$ terminal pairs among $G_\alpha$ and $G_\beta$ an arbitrary, balanced way (the layers receive $\lfloor\frac{a+b-1}{2}\rfloor$ and $\lceil\frac{a+b-1}{2}\rceil$ terminals). For given $u'$ an $u''$ vertices, we may assume, without loss of generality, that $u''\in G_\alpha$. The underlying $H_{u'}$-layer is $(2b-1)$-connected. It contains at most $(n-t+1)$ vertices of horizontal paths and the projection of $u'$ to $G_\beta$. If $(n-t+2)\leq (2b-2)$, we can find a $P_{u'u''}$ path that contains none of the listed vertices.   

If  $(n-t+2)>(2b-2)$, then $s_1=\dots=s_n=1$ or $s_1=\dots=s_{n-1}=1, s_n=2$. These rather simple cases can be handled by very simple case-by case analysis. Choose an empty $H$ layer for every pair of terminals. As each $G$-layer is $(2a-1)$-connected, and there are $(a+b-1)$ pairs of terminals, we can set a path between a terminal $u$ and the assigned  $u'$ endpoint within $G_u$ without entering the other assigned $H$-layers. We join $(u,v)$ by an $(u\--u'\--v'\-- v)$ path. We leave the detailed analysis as an exercise for the reader.  
\end{proof}

We briefly mention that our method with somewhat rougher estimates yields the following variant of Theorem \ref{a+b-1}. This bound is sharp apart from a small ($\leq 6$) constant term for in infinite class of graphs.
 
\begin{theorem}
Suppose $G$ is $a$-linked, $k$-connected graph, $H$ is a $h$-connected graph ( $h\leq k$, $G$ and $H$ are sufficiently large) then $G\square H$ is
 $\frac{a}{2a+1}(k+h)$-linked.
\end{theorem}
\begin{proof} We copy the proof of Theorem \ref{a+b-1}. Let us denote  $L:=\frac{a}{2a+1}(k+h)$. If there exists a crowded $G$-layer (containing at least $(k+1)$ terminals), find a matching pair of terminals (which exists by piegon-hole principal), join them, empty the layer as before and proceed by induction. Otherwise, global shift horizontally, allocate $t:=\lceil\frac{L}{a}\rceil$ empty $G$-layers $G_{\alpha_1},\dots,G_{\alpha_t}$, distribute the terminal pairs among them via vertical paths and reduce the problem to linking within horizontal layers. 
\end{proof}

We believe that the statement of Theorem \ref{a+b-1} is true even without the indicated condition on the minimal size of the graphs (we only assume the condition to guarantee enough room for the shifting techniques). Nevertheless, in the case when $\frac{link(G)}{v(G)}>\frac{1}{8}$ the shifting techniques presented in the main proof fail to work as one has to deal with an aboundance of terminals congested on the layers. Linking of the terminals in that case is likely to lead a rather lenghty and tedious case-by-case analysis involving ac hoc solutions which we do not find particularly interesting.  

\section*{Proof of Theorem \ref{k->k+1}}

Assume we are given the pairing of $(2k+2)$ terminals in $G\square H$. We use the technique of the proof of Theorem \ref{a+b-1} and follow a case-by-case analysis.
\begin{enumerate}
\item If there exist a $G$-layer $G_i$ with $ 3\leq s_i \leq k$ elements, then no $G$-layer is crowded (no $G$-layer contains $2k$ or more terminals). Choose $G_\alpha=G_i$ and apply the horizontal and vertial shift techniques on the remaining $(2k+2)-s_i\leq (2k-1)$ terminals. Observe that $n=(t-1)$, that is, one can use Lemma \ref{classic} in every $G$-layer during the horizontal shift. In the vertival phase the $H$-layer is 2-connected and the path joining $u'$ and $u''$ only has one vertex to avoid (corresponding to $G_\alpha$ or $G_\beta$).

\item If there exist $G$-layers $G_i$ and $G_j$ such that $s_i=1$, $s_j=2$ or $s_i=s_j=2$, choose $G_\alpha=G_i$ and $G_\alpha=G_j$. Solution for the previous case works here as well.
\item If $s_1=\dots=s_{2k+2}=1$, use the separate technique presented for small cases at the end of proof of Theorem \ref{a+b-1}.
\item If $s_1=\dots=s_{n-1}=1$, $k+1\leq s_n\leq 2k-1$, choose $\{G_\alpha, G_\beta\} =\{G_1, G_2\}$ and apply the shifting technique. Lemma \ref{classic} handles every $G$-layer just as in Case 1.  

\item If $s_1=s_2=1$, $s_3=2k$, join a pair $u_1$, $v_1$ within $G_2$ using Lemma \ref{classic} and shift the remaining terminals vertically. If a terminal $u_2$ has no available neighbour, then $v_2\in G_2$ and we can switch pair $(u_1,v_1)$ to $(u_2,v_2)$ and repeat the argument, just as in the crowded layer case of the proof of Theorem \ref{a+b-1}.
\item If $s_1=2$, $s_2=2k$, similar technique works as in Case 5.
\item If $s_n\geq2k+1$, we have all terminals (or all with one exeption) on the same $G_x$-layer for some $x\in H$. Let $y,z\in \Gamma_H(x)$. We can distribute the pairs of terminals between $G_y$ and $G_z$ by using appropriate vertical $\overline{xy}$ and $\overline{xz}$ edges and join $u'$, $v'$ endpoints within the horizontal layer. If $s_n=2k+1$, the missing terminal can be routed to the appropriate layer. We leave the details as an exercise.  
  
\item If $s_1=s_2=k+1$, we may assume none of the layers contain a pair, otherwise we can proceed by matching a pair within a layer, allocating new terminal vertex $u'$ instead of the original terminal $u$ on the layer, shifting and using induction as previously. Let $G_1=G_x$, $G_2=G_y$ for some $x,y\in H$ and let $z\in \Gamma_H(x)-\{y\}$ (as $H$ is 2-connected, such $z$ has to exist). Shift terminals horizontally within $G_x$ if necessarily in order to get for every terminal $u$ a $uu'$ path with endpoint $u'$, such that $H_{u'}$ contains neither a terminal nor a vertex belonging to the shifting paths (in case there was no shift necessary, let $u'=u$). We pick a single terminal $u \in G_x$ and take a path $u-u'-u''$ where $u''$ denotes the projection of $u'$ to $G_y$.  We connect $u''$ with the pair of $u$ in $G_y$ using Lemma \ref{classic}. For the remaining $2k$ pairs, we set a vertical paths for each terminal in $G_x$ and $G_y$ to $G_z$. For a terminal $w\in G_x$ there is no obstacle in $H_w$ to find a path to its projection to $G_z$. If $w\in G_y$, we use the fact the $H$ is 2-connected and that it contains at most one vertex of $G_x$ we might have used previously. Having set the vertical paths, we join the projections in $G_z$.

\end{enumerate} 

Assuming that the product $G\square H$ is $k$-linked yields no essential lower bound neither on linkedness nor on connectivity of $G$ or $H$. The theorem of Bollob\'as and Thomason \cite{Bollobas} together with the result of Spacapan \cite{conn}  show that large degree is sufficient to imply high linkedness while the component graphs are connected but might not even be 2-connected. That is, there exist a function f such that $\delta(G)+\delta(H)\geq f(k)$ implies $G\square H$ is $k$-linked. Using the improved bound presented in \cite{10k} we know that $f(k)\leq 10k$ and detailed analysis might yield even better bounds.

\section*{Linkedness of Hypercubes, Products of Paths and Products of Cycles}
We determine the linkedness-number of the $n$-dimensional grid, that is, the Cartesian product of $n$ paths, and the linkedness of product of  $n$ cycles. We use a straightforward corollary of Theorem \ref{k->k+1}:

\begin{corollary}\label{cycle}
If $G$ is a $k$-linked graph, $k\geq 2$, $|G|\geq \max(9, 4k)$ then $G\square C_m$ is $(k+1)$-linked, where $C_m$ denotes the cycle of length $m$.  
\end{corollary}
\begin{lemma}\label{C_m x C_n}
For cycles of length $m$ and $n$ ($m,n\geq 3$) $link (C_m\square C_n)=2$. 
\end{lemma}
\begin{proof}
It can be showed by a simple but rather lengthy case-by-case analysis that $C_3\square C_3$, $C_3\square C_4$ and $C_4\square C_4$ are 2-linked. If $\max(m,n)\geq 5$, one of the cycles can be shortened by substituting an empty layer with vertical / horizontal edges joining its neighbours and proceed by induction. 
\end{proof}
\begin{proposition}
 For cycles of length $m_1,\dots,m_t$ ($m_i\geq 3$, $t\geq 2$) $link( C_{m_1}\square\dots\square C_{m_t})=t$.  
\end{proposition}
It follows directly from Corollary \ref{cycle} and Lemma \ref{C_m x C_n}.
\begin{proposition}
Let $Q_n$ denote the $n$-dimensional hypercube.  $link(Q_n)=\lceil\frac{n}{2}\rceil$ if $n\neq 3$.
\end{proposition}
As $Q_n$ is $n$-connected, the linkedness number of $Q_n$ is  at most $\lceil\frac{n}{2}\rceil$. 
Equality holds for $n=1$ and $2$. $Q_3$ is not 2-linked as being a planar graph with non-triangle faces. $Q_4$ is 2-linked.
\begin{lemma}
For the five dimensional hypercube $link(Q_5)=3$.
\end{lemma}
\begin{proof} We distinguish two cases:
\begin{description}
\item{Case 1} Assume there exist terminals $x_1,y_1$ satisfying $d(x_1,y_1)\leq 4$. Because of symmetries of the graph $Q_5$, we may assume without loss of generality that $x_1=(0,0,0,0,0)$ and $y_1=(\underline{v},0)$, $\underline{v}\in Q_4$. Also, let us denote $Q_5=Q_4^0\cup Q_4^1$, the decomposition of $Q_5$ into affine hyperplanes being isomorphic to $Q_4$ (with respect to the last coordinate). Certainly, $x_1,y_1\in Q_4^0$ and we may assume that $Q_4^0-x_2-y_2-x_3-y_3$ is connected (otherwise switch to pair $(x_2,y_2)$). Join $x_1$ to $y_1$ in $Q_4$ by any path of length 4 encountering no other terminal. We want to join the remaining two pairs in $Q_4^1$. If a terminals $u\in \{x_2, y_2, x_3, y_3\}$ lies in $Q_4^0$, we define a crossing path that ends at $u'\in Q_4^1$. If the projection of $u$ to $Q_4^1$ is not a terminal vertex (of if it happens to be  the pair of $u$), we take that very edge as the required path. In every other case there is a $v\in Gamma_{Q_4^0}(u)$ such that the projection of $v$ to $Q_4^1$ is available, yielding an appropriate path of length 2.

\item{Case 2} If $d(x_1,y_1)=d(x_2,y_2)=d(x_3,y_3)=5$, there exist - up to isomorphism - 5 possible arrangements of the terminals. We leave the easy case-by-case analysis to the reader. 
\end{description}
\end{proof}
As $Q_n=Q_{n-2}\square C_4$ , Corollary \ref{cycle} applies (for $n\geq 4$) and so the proof is complete.
\begin{proposition}
Let $P_m$ denote the path of $m$ vertices and let $G= P_{m_1}\square\dots\square P_{m_t}$. Then we have
\begin{enumerate}
\item[i)] $link(G)=1$ if $t=3$, $m_1=m_2=2$ and
\item[ii)] $link(G)=\lceil\frac{t}{2}\rceil$ if $t\neq 3$, $m_i\geq 2$ or $t=3$, $m_3\geq m_2\geq 3$.
\end{enumerate}
\end{proposition}
\begin{proof}
The first statement is obvious as $G$ is a planar graph. For $t\neq 3$, let $Q_n$ be an induced subgraph of $G$ containing terminals $x_1,\dots,x_p$, $p\geq 1$. As $G-x_2-\dots-x_p$ is $t-p$-connected, the set of remaining terminals can be routed to $Q_n$ and linking can be performed. The case $t=3$, $m_3\geq m_2\geq 2$ can be solved by the previous idea using the fact that $P_2\square P_3\square P_3$ is 2-linked.
\end{proof}

\section*{Acknowledgement}
The author wishes to thank Professor Ervin Gy\H ori for his  helpful suggestions, interest and valuable guidance.  The author  is also grateful for the Balassi Institute, the Fulbright Commission and the Rosztoczy Foundation for their kind and generous support.

\end{document}